\documentclass[12pt]{amsart}
\usepackage{fourier}
\usepackage[T1]{fontenc}
\usepackage{graphics}
\usepackage{amsfonts,amssymb,color}
\usepackage[mathscr]{eucal}
\usepackage{amsmath, amsthm}
\usepackage{mathrsfs}
\usepackage{amsbsy}
\usepackage{dsfont}
\usepackage{bbm}
\usepackage{wasysym}
\usepackage{stmaryrd}
\usepackage{url}

\input xypic
\xyoption{all}

\makeindex
\makeglossary

\begin{document}
\baselineskip = 16pt

\newcommand \ZZ {{\mathbb Z}}
\newcommand \NN {{\mathbb N}}
\newcommand \RR {{\mathbb R}}
\newcommand \PR {{\mathbb P}}
\newcommand \AF {{\mathbb A}}
\newcommand \GG {{\mathbb G}}
\newcommand \QQ {{\mathbb Q}}
\newcommand \bcA {{\mathscr A}}
\newcommand \bcC {{\mathscr C}}
\newcommand \bcD {{\mathscr D}}
\newcommand \bcF {{\mathscr F}}
\newcommand \bcG {{\mathscr G}}
\newcommand \bcH {{\mathscr H}}
\newcommand \bcM {{\mathscr M}}
\newcommand \bcJ {{\mathscr J}}
\newcommand \bcL {{\mathscr L}}
\newcommand \bcO {{\mathscr O}}
\newcommand \bcP {{\mathscr P}}
\newcommand \bcQ {{\mathscr Q}}
\newcommand \bcR {{\mathscr R}}
\newcommand \bcS {{\mathscr S}}
\newcommand \bcU {{\mathscr U}}
\newcommand \bcV {{\mathscr V}}
\newcommand \bcW {{\mathscr W}}
\newcommand \bcX {{\mathscr X}}
\newcommand \bcY {{\mathscr Y}}
\newcommand \bcZ {{\mathscr Z}}
\newcommand \goa {{\mathfrak a}}
\newcommand \gob {{\mathfrak b}}
\newcommand \goc {{\mathfrak c}}
\newcommand \gom {{\mathfrak m}}
\newcommand \gon {{\mathfrak n}}
\newcommand \gop {{\mathfrak p}}
\newcommand \goq {{\mathfrak q}}
\newcommand \goQ {{\mathfrak Q}}
\newcommand \goP {{\mathfrak P}}
\newcommand \goM {{\mathfrak M}}
\newcommand \goN {{\mathfrak N}}
\newcommand \uno {{\mathbbm 1}}
\newcommand \Le {{\mathbbm L}}
\newcommand \Spec {{\rm {Spec}}}
\newcommand \Gr {{\rm {Gr}}}
\newcommand \Pic {{\rm {Pic}}}
\newcommand \Jac {{{J}}}
\newcommand \Alb {{\rm {Alb}}}
\newcommand \Corr {{Corr}}
\newcommand \Chow {{\mathscr C}}
\newcommand \Sym {{\rm {Sym}}}
\newcommand \Prym {{\rm {Prym}}}
\newcommand \cha {{\rm {char}}}
\newcommand \eff {{\rm {eff}}}
\newcommand \tr {{\rm {tr}}}
\newcommand \Tr {{\rm {Tr}}}
\newcommand \pr {{\rm {pr}}}
\newcommand \ev {{\it {ev}}}
\newcommand \cl {{\rm {cl}}}
\newcommand \interior {{\rm {Int}}}
\newcommand \sep {{\rm {sep}}}
\newcommand \td {{\rm {tdeg}}}
\newcommand \alg {{\rm {alg}}}
\newcommand \im {{\rm im}}
\newcommand \gr {{\rm {gr}}}
\newcommand \op {{\rm op}}
\newcommand \Hom {{\rm Hom}}
\newcommand \Hilb {{\rm Hilb}}
\newcommand \Sch {{\mathscr S\! }{\it ch}}
\newcommand \cHilb {{\mathscr H\! }{\it ilb}}
\newcommand \cHom {{\mathscr H\! }{\it om}}
\newcommand \colim {{{\rm colim}\, }} 
\newcommand \End {{\rm {End}}}
\newcommand \coker {{\rm {coker}}}
\newcommand \id {{\rm {id}}}
\newcommand \van {{\rm {van}}}
\newcommand \spc {{\rm {sp}}}
\newcommand \Ob {{\rm Ob}}
\newcommand \Aut {{\rm Aut}}
\newcommand \cor {{\rm {cor}}}
\newcommand \Cor {{\it {Corr}}}
\newcommand \res {{\rm {res}}}
\newcommand \red {{\rm{red}}}
\newcommand \Gal {{\rm {Gal}}}
\newcommand \PGL {{\rm {PGL}}}
\newcommand \Bl {{\rm {Bl}}}
\newcommand \Sing {{\rm {Sing}}}
\newcommand \spn {{\rm {span}}}
\newcommand \Nm {{\rm {Nm}}}
\newcommand \inv {{\rm {inv}}}
\newcommand \codim {{\rm {codim}}}
\newcommand \Div{{\rm{Div}}}
\newcommand \sg {{\Sigma }}
\newcommand \DM {{\sf DM}}
\newcommand \Gm {{{\mathbb G}_{\rm m}}}
\newcommand \tame {\rm {tame }}
\newcommand \znak {{\natural }}
\newcommand \lra {\longrightarrow}
\newcommand \hra {\hookrightarrow}
\newcommand \rra {\rightrightarrows}
\newcommand \ord {{\rm {ord}}}
\newcommand \Rat {{\mathscr Rat}}
\newcommand \rd {{\rm {red}}}
\newcommand \bSpec {{\bf {Spec}}}
\newcommand \Proj {{\rm {Proj}}}
\newcommand \pdiv {{\rm {div}}}
\newcommand \CH {{\it {CH}}}
\newcommand \wt {\widetilde }
\newcommand \ac {\acute }
\newcommand \ch {\check }
\newcommand \ol {\overline }
\newcommand \Th {\Theta}
\newcommand \cAb {{\mathscr A\! }{\it b}}

\newenvironment{pf}{\par\noindent{\em Proof}.}{\hfill\framebox(6,6)
\par\medskip}

\newtheorem{theorem}[subsection]{Theorem}
\newtheorem{conjecture}[subsection]{Conjecture}
\newtheorem{proposition}[subsection]{Proposition}
\newtheorem{lemma}[subsection]{Lemma}
\newtheorem{remark}[subsection]{Remark}
\newtheorem{remarks}[subsection]{Remarks}
\newtheorem{definition}[subsection]{Definition}
\newtheorem{corollary}[subsection]{Corollary}
\newtheorem{example}[subsection]{Example}
\newtheorem{examples}[subsection]{examples}

\title{Algebraic cycles on the Fano variety of lines of a cubic fourfold }
\author{Kalyan Banerjee}

\address{Tata Institute of Fundamental Research, Mumbai, India}

\email{kalyan@math.tifr.res.in}

\begin{abstract}
  In this text we prove that if a smooth cubic in $\PR^5$ has its Fano variety of lines birational to the Hilbert scheme of two points on a K3 surface, then there exists a smooth projective curve or a smooth projective surface embedded in the Fano variety, such that the kernel of the push-forward  (at the level of zero cycles ) induced  by the closed embedding is torsion.
\end{abstract}
\maketitle

\section{Introduction}
In the article \cite{BG} the authors were discussing the kernel of the push-forward homomorphism at the level of algebraically trivial one cycles modulo rational equivalence, induced by the closed embedding of a smooth hyperplane section of a cubic fourfold into the cubic itself. It has been proved that the kernel of this homomorphism is countable when we consider a very general hyperplane section. It was a natural question at that time that can we formulate the rationality problem of a very general cubic fourfold in terms of this kernel. That is what is the obstruction to rationality in terms of the kernel.

In \cite{GS}[theorem 7.5] it has been proved that if the ground field $k$ is such that the class of the affine line is not a zero divisor in the Grothendieck ring of varieties over $k$ and if the cubic fourfold is rational then the Fano variety of lines on the cubic fourfold is birational to the Hilbert scheme of two points on a K3 surface. This is the starting points of this paper. Although it is a point to be noted that this assumption is false, \cite{BO}, when the ground field is $\mathbb C$. We start with a birational map from the Hilbert scheme of two points on a K3 surface to theh Fano variety of lines and try to formulate the cycle-theoretic consequence of it,  in terms of zero cycles on the Fano variety. It is also known, due to Hassett \cite{Ha}  that for a certain class of rational cubic fourfolds the Hilbert scheme is actually isomorphic to the Fano variety of lines.

To proceed in this direction of studying  the zero cycles on the Fano variety of lines, we study algebraic cycles on symmetric powers of a smooth projective variety following Collino \cite{Collino}, where the algebraic cycles on symmetric powers of a smooth projective curve was studied. It has been proved in \cite{Collino} that the closed embedding of one symmetric power of a smooth projective curve into a bigger symmetric power of the same curve induces injective push-forward at the level of Chow groups. We follow this approach and try to understand the algebraic cycles on the blowup of $\Sym^2 S$ along the diagonal where $S$ is a smooth projective surface. This blowup is the Hilbert scheme of two points on a smooth projective surface. Following Collino's approach we study a correspondence between the Hilbert scheme of two points on a surface and the strict transform the surface itself under the blowup along the diagonal. Our main result of this paper is the following:

\textit{Let $X$ be a smooth cubic fourfold in $\PR^5$ and let $F(X)$ be the Fano variety of lines on the cubic fourfold $X$. Let the cubic $X$ is such that there exists a birational map from the Hilbert scheme of two points on a fixed K3 surface to the Fano variety of $X$ and let the indeterminacy locus of the rational map $\Hilb^2 S\to F(X)$ is smooth. Then there exists a non-rational curve $D$ in $F(X)$ such that  the push-forward induced by the inclusion of $D$ into $F(X)$ has torsion kernel.}


{\small \textbf{Acknowledgements:} The author wishes to thank Vladimir Guletskii for suggesting the idea of studying the relation between torsion elements in Chow groups and non-rationality of cubics in $\PR^5$. The author thanks the anonymous referee for his careful reading of the manuscript and for suggesting relevant modifications. The author is also grateful to Marc Levine for pointing out a technical mistake in the proof of theorem 2.1 and to V.Srinivas and C.Voisin for  useful discussions relevant to the theme of the paper. Author is indebted to Mingmin Shen for pointing out a mistake in section 4. The author remains grateful to B.Hassett for improving the conclusion of the main theorem. Finally the author also wishes to thank the ISF-UGC grant for funding this project and hospitality of Indian Statistical Institute, Bangalore Center for hosting this project.}

\section{Algebraic cycles on Hilbert schemes}
In this section we are going to prove the following. Let $S$ be a smooth, projective algebraic surface. Consider the diagonal embedding of $S$ into $\Sym^2 S$. Then the blow up of $\Sym^2 S$ along $S$ is the Hilbert scheme of two points on $S$. Let $\wt{\Sym^2 S}$ denote this Hilbert scheme and let $\wt{S}$ denote the inverse image of $S$ in $\wt{\Sym^2 S}$ (This copy of $S$ is different from the image of the diagonal embedding and it is the image of the closed embedding $s\mapsto [s,p]$). Then the embedding of $\wt{S}$ into $\wt{\Sym^2 S}$ induces an injective push-forward homomorphism at the level of Chow group of zero cycles.

To prove that we prove that if blow up $\Sym^n X$ along a subvariety $Z$, which does not contain a copy of $\Sym^m X$ for $m\leq n$, then the closed embedding of the inverse image of $\Sym^m X$ into the blow up of $\Sym^n X$ induces injective push-forward homomorphism at the level of Chow groups of zero cycles.

\begin{theorem}
\label{theorem 1}
Let $X$ be a smooth projective variety. Let $\wt{\Sym^n X}$ be the blow up of $\Sym^n X$ along a smooth subvariety $Z$, such that the blow-up will be a smooth projective variety. Let $\Sym^m X$ intersect  $Z$ transversally. Then the closed embedding of the strict transform  of $\Sym^m X$ into $\wt{\Sym^n X}$ induces an injective push-forward homomorphism at the level of Chow groups of zero cycles.
\end{theorem}

\begin{proof}
To prove the theorem we follow the approach as in \cite{Collino}. First consider the correspondence $\Gamma$ on $\Sym^n X\times \Sym^m X$ given by
$$(\pi_n\times \pi_m)(Graph(pr))$$
where $pr$ is the projection from $X^n$ to $X^m$ and $\pi_i$ is the natural quotient morphism from $X^i$ to $\Sym^i X$. Let $f$ be the natural morphism from $\wt{\Sym^n X}$ to $\Sym^n X$ and $f'$ be the restriction of $f$ to $\wt{\Sym^m X}$, which is the inverse image of $\Sym^m X$. Consider the correspondence $\Gamma'$ given by $(f\times f')^*(\Gamma)$. Then a computation following \cite{Collino} shows that $\Gamma'_*j_*$ is induced by $(j\times id)^*(\Gamma')$, where $j$ is the closed embedding of $\wt{\Sym^m X}$ into $\wt{\Sym^n X}$. Here we can consider the pull-back by $f\times f'$ and by $j\times id$, because they are local complete intersection morphisms since we are considering monoidal transformations \cite{Fulton}[page 332, section 17.5]. Consider the following diagram.

$$
  \diagram
   \wt{Sym^m X}\times \wt{\Sym^m X}\ar[dd]_-{f'\times f'} \ar[rr]^-{j\times id} & & \wt{\Sym^n X}\times \wt{\Sym^m X} \ar[dd]^-{f\times f'} \\ \\
  \Sym^m X\times \Sym^m X \ar[rr]^-{i\times id} & & \Sym^n X\times \Sym^m X
  \enddiagram
  $$
This diagram gives us the following at the level of algebraic cycles.
$$(j\times id)^*(f\times f')^*(\Gamma)=(f'\times f')^*(i\times id)^*(\Gamma)$$
that is equal to
$$(f'\times f')^*(\Delta+Y)$$
by \cite{Collino}, where $Y$ is supported on $\Sym^m X\times \Sym^{m-1}X$, and $\Delta$ is the diagonal of $\Sym^m X\times \Sym^m X$.
Now
$$(f'\times f')^*(\Delta)=\Delta+E_1$$
where
$$E_1\subset(E\times E)\cap (\wt{\Sym^m X}\times \wt{\Sym^m X}) $$
 $E$is the exceptional locus, that is the pre-image of $Z$.
So then by the Chow moving lemma, we can take the support of a zero cycle away from $E\cap \wt{\Sym^m X}$, which is a proper closed subscheme of $\wt{\Sym^m X}$, and we get
$$\rho^*\Gamma'_*j_*(z)=\rho^*(z+z_1)=\rho^*(z)$$
where $\rho$ is the open embedding of the complement of $ \wt{\Sym^{m-1}X}$ in $\wt{\Sym^m X}$, and $z_1$ is supported on
$\wt{\Sym^{m-1}X}$. Now consider the following commutative diagram with the rows exact.

$$
  \xymatrix{
     \CH^*(\wt{\Sym^{m-1}X}) \ar[r]^-{j'_{*}} \ar[dd]_-{}
  &   \CH^*(\wt{\Sym^m X}) \ar[r]^-{\rho^{*}} \ar[dd]_-{j_{*}}
  & \CH^*(U)  \ar[dd]_-{}  \
  \\ \\
   \CH^*(\wt{\Sym^{m-1}X}) \ar[r]^-{j''_*}
    & \CH^*(\wt{\Sym^{n} X}) \ar[r]^-{}
  & \CH^*(V)
  }
$$
Suppose that $j_*(z)=0$, by the above discussion we have
$\rho^*\Gamma'_*j_*(z)=\rho^*(z)=0$. Hence by the localisation exact sequence there exists $z'$ such that $j'_*(z')=z$.  So by induction $\CH_*(\wt{\Sym^{m-1}X})$ to $\CH_*(\wt{\Sym^n X})$ is injective.

So we have $z'=0$ and hence $z=0$. So we have the theorem.

\end{proof}

Let $S$ be a smooth projective algebraic surface. Let us consider the diagonal embedding of $S$ into $\Sym^2 S$ and the embedding of $S$ into $\Sym^2 S$ given by $s\mapsto [s,p]$, where $p$ is a fixed closed point on $S$. Let $\wt{\Sym^2 S}$ be the blow up of $\Sym^2 S$ along the image of the diagonal embedding and let $\wt{S}$ be the strict transform of the other copy of $S$ prescribed by the late embedding.

\begin{corollary}
The closed embedding of $\wt{S}$ into $\wt{\Sym^2 S}$ induces injective push-forward homomorphism at the level of Chow groups of zero cycles.
\end{corollary}
\begin{proof}
Follows from \ref{theorem 1}.
\end{proof}

$\wt{\Sym^2 S}$ is nothing but the Hilbert scheme of two points on $S$.
The above corollary tell us in particular that the closed embedding of $\wt{S}$ into the Hilbert scheme of two points on a surface $S$ induces injective push-forward homomorphism at the level of Chow group of zero cycles.

Using the same technique we can prove that if we blow up the Hilbert scheme once again and consider the total transform of the surface $\wt{S}$, then the closed embedding of the total transform in the blow-up gives rise to an injective push-forward homomorphism at the level of Chow group of zero cycles.

\section{Blow down of symmetric powers}
In this section we are going to consider blow downs of $\Sym^n X$, for a smooth projective variety $X$, and consider the images of $\Sym^m X$ under the blow down. Then we want to investigate the nature of the push-forward homomorphism induced at the level of Chow groups by the closed embedding of the blow down of $\Sym^m X$ into the blow down of $\Sym^n X$.

\begin{theorem}
\label{theorem2}
Let $\wt{\Sym^n X}$ be a smooth blow down of $\Sym^n X$ and let $\wt{\Sym^m X}$ denote the image of $\Sym^m X$ under this blow down. Then the closed embedding of  $\wt{\Sym^m X}$ into $\wt{\Sym^n X}$ induces an injective push-forward homomorphism at the level of Chow groups.
\end{theorem}
\begin{proof}
As before consider the correspondence $\Gamma$ given by $(\pi_n\times \pi_m)(Graph(pr))$, where $pr$ is the projection from $X^n$ to $X^m$ and $\pi_i$ is the natural morphism from $X^i$ to $\Sym^i X$. Let $\wt{f}$ denote the morphism from $\Sym^n X$ to $\wt{\Sym^n X}$, and $f'$ its restriction to $\wt{\Sym^m X}$. Consider the correspondence $\Gamma'$ on $\wt{\Sym^n X}\times \wt{\Sym^m X}$ given by $(f\times f')_*(\Gamma)$. Then following \cite{Collino}, we have that, $\Gamma'_*j_*$ is induced by $(j\times id)^*(\Gamma')$, where $j$ is the closed embedding of $\wt{Sym^m X}$ into $\wt{\Sym^n X}$. Now consider the following Cartesian square.
$$
  \diagram
   {Sym^m X}\times {\Sym^m X}\ar[dd]_-{f'\times f'} \ar[rr]^-{i\times id} & & {\Sym^n X}\times {\Sym^m X} \ar[dd]^-{f\times f'} \\ \\
  \wt{\Sym^m X}\times \wt{\Sym^m X} \ar[rr]^-{j\times id} & & \wt{\Sym^n X}\times \wt{\Sym^m X}
  \enddiagram
  $$

Then since we have above Cartesian square and $f,f'$ are birational, we have that
$$(j\times id)^*\Gamma'=(j\times id)^*(f\times f')_*(\Gamma)$$
$$=(f'\times f')_*(i\times id)^*(\Gamma)=(f'\times f')_*(\Delta+Y)$$
Here $i$ is the closed embedding of $\Sym^m X$ into $\Sym^n X$, $\Delta$ is the diagonal of $\Sym^m X\times \Sym^m X$, and $Y$ is supported on $\Sym^m X\times \Sym^{m-1}X$. Now
$$(f'\times f')_*(\Delta)=\Delta$$
the diagonal of $\wt{\Sym^m X}\times \wt{\Sym^m X}$ and $(f'\times f')_*(Y)$ is supported on $\wt{\Sym^m X}\times \wt{\Sym^{m-1}Y}$. The multiplicity of the diagonal of $\wt{\Sym^m X}\times \wt{\Sym^m X}$ is $1$, because the morphism $f$ from $\Sym^n X$ to $\wt{\Sym^n X}$ is birational and the multiplicity of $\Delta$ of $\Sym^m X\times \Sym^m X$ is $1$.

Let $\rho$ be the open embedding of the complement of $\wt{\Sym^{m-1}X}$ into $\wt{\Sym^m X}$. Then the above computation shows that
$$\rho^*\Gamma'_*j_*(z)=\rho^*(z+z_1)=\rho^*(z)$$
where $z_1$ is supported on $\wt{\Sym^{m-1}X}$. Now consider the following commutative square with the rows exact.
$$
  \xymatrix{
     \CH^*(\wt{\Sym^{m-1}X}) \ar[r]^-{j'_{*}} \ar[dd]_-{}
  &   \CH^*(\wt{\Sym^m X}) \ar[r]^-{\rho^{*}} \ar[dd]_-{j_{*}}
  & \CH^*(U)  \ar[dd]_-{}  \
  \\ \\
   \CH^*(\wt{\Sym^{m-1}X}) \ar[r]^-{j''_*}
    & \CH^*(\wt{\Sym^{n} X}) \ar[r]^-{}
  & \CH^*(V)
  }
$$
Suppose that $j_*(z)=0$, then we have $\rho^*\Gamma'_*j_*(z)=0$, that gives us $\rho^*(z)=0$. By the localisation exact sequence we have that $z=j'_*(z')$. By the commutativity of the above diagram we have that
$j''_*(z')=0$ and by induction hypothesis we have that $z'=0$ and hence $z=0$. So the homomorphism $j_*$ is injective.

\end{proof}

\section{Fano variety of lines on a cubic fourfold and algebraic cycles}
In this section we consider the  cubic fourfolds and the Fano variety of lines on them.  Suppose that the Fano variety of lines admit of a bi-rational map from the Hilbert scheme of two points on a K3 surface $S$. The precise theorem is as follows.



\begin{theorem}
\label{theorem4}
Let $X$ be a cubic fourfold. Let $F(X)$ denote the Fano variety of lines on $X$. Suppose that there exists a birational, dominant map from $\Hilb^2(S)$, the Hilbert scheme of two points on a fixed surface $S$, to $F(X)$. Let $\wt{S}$ be the strict transform of $S$, under the Blow up from $\Sym^2 S$ to $\Hilb^2(S)$,which intersect the indeterminacy locus of the above rational map transversally. Then we have a codimension $2$ subvariety $D$ in $F(X)$,  such that the kernel of the pushforward $\CH_0(D)$ to $\CH_0(F(X))$ is torsion.
\end{theorem}
\begin{proof}
Let us consider the rational map from $\Hilb^2(S)$ to $F(X)$. Let us resolve the indeterminacy of the rational map and we get a regular map
$$\widehat{\Hilb^2(S)}\to F(X)$$
where $\widehat{\Hilb^2(S)}$ denote the Blow up of $\Hilb^2(S)$ along the indeterminacy locus. Suppose that this locus is smooth and intersect a copy of $\wt{S}$ transvesally. Then first we prove that the closed embedding $\widehat{S}\to \widehat{\Hilb^2(S)}$ gives an injective push-forward homomorphism at the level of Chow groups of zero cycles, here $\widehat{S}$ is the strict transform of $\wt{S}$ under this blow-up.
Let $f_1$ be the natural morphism from $\widehat{\Hilb^2(S)}$ to $\Hilb^2(S)$ and $f_1'$ be the restriction of it to $\widehat{S}$. Consider the correspondence $\Gamma''$ on $\widehat{\Hilb^2(S)}\times \widehat{S}$ given by $(f_1\times f_1')^*(\Gamma')$. Let $\hat{j}$ denote the closed embedding of $\widehat{S}$ into $\widehat{\Hilb^2(S)}$. Then as previous we have that $\Gamma''_*\hat{j_*}$ is induced by $(\hat{j}\times id)^*(\Gamma'')$, which is equal to
$$(\hat{j}\times id)^*(f_1\times f_1')^*\Gamma'$$
which is
$$(f'_1\times f'_1)^*(\Delta+Y_1+E_1)$$
where $E_1$ is the intersection of the exceptional locus of the blow up $\Hilb^2(S)\to \Sym^2 S$ with $\wt{S}$ and $Y_1$ is supported on $\wt{S}\times \PR^1$. Then we have
$$(f'_1\times f'_1)^*(\Delta+Y_1+E_1)=\Delta+\hat{E}_2+\hat{E_1}+\hat{Y_1}\;.$$
Here $\hat{E}_2$ is the intersection of the exceptional locus $\hat{E}$ of the blow up $\widehat{\Hilb^2(S)}\to \Hilb^2(S)$ with $\hat{S}$. Then consider the open embedding of the complement of $ \hat {Y}$ in $\hat{S}$ (where $\hat{Y}$ is a  $\PR^1$-bundle over $\PR^1$, or a rational curve). Call it $\rho$. We have by Chow moving lemma,
$$\rho^*\Gamma''_*\hat{j}_*(z)=\rho^*(z+z_1)=\rho^*(z)\;.$$
Here $z_1$ is supported on $A$. Consider the following commutative diagram where the rows are exact.
$$
  \xymatrix{
     \CH^*(\hat{Y}) \ar[r]^-{j'_{*}} \ar[dd]_-{}
  &   \CH^*(\hat{S}) \ar[r]^-{\rho^{*}} \ar[dd]_-{\hat{j}_{*}}
  & \CH^*(U)  \ar[dd]_-{}  \
  \\ \\
   \CH^*(\hat{Y}) \ar[r]^-{j''_*}
    & \CH^*(\widehat{\Hilb^2(S)}) \ar[r]^-{}
  & \CH^*(V)
  }
$$
Suppose that $\hat{j}_*(z)=0$, then we have $\rho^*\Gamma''_*\hat{j}_*(z)=0$, which gives $\rho^*(z)=0$. So there exists $z'$ such that $j_*(z')=z$. The map from $\CH_0(\hat{Y})\to \CH_0(\widehat{\Hilb^2 S})$ is injective. So $z'=0$, hence $z=0$.
Now we have a regular  (generically finite) map from $\widehat{\Hilb^2 (S)}$ to $F(X)$ and let $D$ denote the  image of $\hat{S}$.

First we  prove that $D$ is a surface, that is the dimension does not drop under the blow-down from $\widehat{\Hilb^2 (S)}$ to $F(X)$.
For that consider the universal family
$$\bcU=\{([x,y],p)|p\in [x,y]\}\subset \Sym^2 S\times S$$
consider its pull-back to $\Hilb^2 S$ and further to $\widehat{\Hilb^2(S)}$. Consider the projection from $\bcU\to S$. For $p\in S$, we have the fiber of the projection over $p$ is $S_p$, that is a copy of $S$ in $\Sym^2 S$ embedded by $x\mapsto [x,p]$. Now given any $[x,y]$ in $\Sym^2 S$, it belongs to some $S_p$. So the $S_p$'s cover $\Sym^2 S$. Therefore the strict transform of $S_p$'s, under the blow up from $\Sym^2 S$ to $\Hilb^2(S)$, cover $\Hilb^2(S)$. Hence there exists one $S_p$, such that $\wt{S_p}$ is not contained in the indeterminacy locus of the rational map from $\Hilb^2(S)$ to $F(X)$. Moreover we can prove that the collection of points in $S$ such that $S_p$ is not contained in the above mentioned indeterminacy locus is Zariski closed. Similarly the strict transforms $\hat{S_p}$ cover $\widehat{\Hilb^2(S)}$, hence there exists one $\hat{S_p}$, which is not contained in the blow-down locus of the map $\widehat{\Hilb^2(X)}\to F(X)$. The collection of $p$, such that $\hat{S_p}$ is contained in the blow-down locus is Zariski closed. Hence there exists $p$ such that $\wt{S_p}$ is not contained in  the indeterminacy locus of the rational map $\Hilb^2(S)\to F(X)$ and also not contained in the center of blow down of $\widehat{\Hilb^2(S)}\to F(X)$.  The image of that $\hat{S_p}$ is a surface. That is our $D$.

Then we prove that the kernel of the push-forward homomorphism from $\CH_0(D)$ to $\CH_0(F(X))$ is torsion. So let $h$ be the closed embedding of $D$ into $F(X)$ and $g$ the map from $\widehat{\Hilb^2(S)}$ to $F(X)$ and $g'$ its restriction to $\hat{S}$. Then consider the correspondence $\Gamma'''$ to be $(g\times g')_*(\Gamma'')$. The homomorphism $\Gamma'''_*h_*$ is induced by the cycle $(h\times id)^*(\Gamma''')$, which is
$$(h\times id)^*(g\times g')_*(\Gamma'')$$
and is equal to the fundamental cycle \cite{Fulton}[1.5] associated to (since $g$ is birational and generically finite)
$$(g'\times g')(\hat{j}\times id)^{-1}(\Gamma'')+\hat{E_3}$$
where $E_3$ is supported on the  product $(D\cap B)\times (D\cap B)$, $B$ is the image of the blown down variety.
So the above is equal to
$$(d\Delta+(g'\times g')_*(\hat{E}_2+\hat{E_1}+\hat{Y}))+\hat{E_3}$$
where $(g'\times g')(\hat{Y})$ is supported on $D\times V$ where $V$ is the image of the $\PR^1$ bundle $\hat{Y}$ or on $D\times C$ where $C$ is a rational curve (possibly singular) on $D$. So let $\rho$ be the open embedding of the complement of $g'(\hat{Y})$ in $D$. Then as previous we have that
$$\rho^*\Gamma'''_*h_*(z)=\rho^*(dz+z_1)$$

where $z_1$ is supported on $g'(\hat{Y})$ (this is obtained by Chow moving lemma). Then consider the following commutative diagram with the rows exact.
$$
  \xymatrix{
     \CH^*(g'(\hat{Y})) \ar[r]^-{h'_{*}} \ar[dd]_-{}
  &   \CH^*(D) \ar[r]^-{\rho^{*}} \ar[dd]_-{h_{*}}
  & \CH^*(U)  \ar[dd]_-{}  \
  \\ \\
   \CH^*(g'(\hat{Y})) \ar[r]^-{j''_*}
    & \CH^*(F(X)) \ar[r]^-{}
  & \CH^*(V)
  }
$$
Then if $h_*(z)=0$, it would follow that $d(\rho^*(z))=0$. So we have $h_*'(z')=dz$. Now $\CH_0(g'(\hat{Y}))\to \CH_0(F(X))$ is injective as $V$ is rational or $\CH_0(g'\hat{Y})$ is torsion since $C$ is rational, tell us that $z'=0$ or $d'z'=0$ so $dz=0$ or $dd'z=0$. So the kernel of $h_*$ is torsion.

\end{proof}

Now all this techniques can be applicable in the following setup. Let us embed $\Sym^2 S$ into some projective space. Let $C$ be a smooth hyperplane section of $S$ inside $\Sym^2 S$ such that $C$ intersects the diagonal transversally (or do not intersect). Then consider the strict transform of $\wt{C}$ inside $\wt{\Sym^2 S}=\Hilb^2(S)$, that gives us  push-forward homomorphism at the level of Chow group of zero cycles, which has torsion kernel. When we blow up $\Hilb^2 S$, then we get the strict transform $\hat{C}$, which is again inducing a  push-forward at the level of Chow groups from $\CH_0(\hat{C})$ to $\CH_0(\widehat{\Hilb^2(S)})$ having torsion kernel. Then we push-down everything to $F(X)$, we get that the image of $\hat{C}$, inside $F(X)$, say $C'$, is such that the push-forward $\CH_0(C')\to \CH_0(F(X))$ has torsion kernel. Now Consider an embedding of $\Sym^2 S$ into $\PR^N$ such that the general hyperplane section of $S$ is non-rational. Now consider the family
$$\bcH=\{([x,y],p,t)|p\in [x,y], [x,y]\in \Sym^2 S\cap H_t\}\subset \Sym^2 S\times S\times {\PR^N}^*$$
Now given any $p$ in $S$, the fiber of the projection $\bcH\to S$ is nothing but the family of hyperplane sections of $S_p$. A general such hyperplane section is non-rational, call it $C_t$. As in the previous Theorem \ref{theorem4}, there exists $p$, such that $\hat{S_p}$ is not contained in the exceptional divisor of the blow up $\widehat{\Hilb^2(S)}\to \Hilb^2(S)$ and also not in the center of blow down of the blow down $\widehat{\Hilb^2(S)}$ to $F(X)$. Consider a general hyperplane section of $S_p$, which intersects the diagonal of $\Sym^2 S$ properly. Then the strict transform $\hat{C_t}$ under two successive blow-ups of such a $C_t$ is non-rational. Also since $\hat{C_t}$ is not contained in the center of blow-down of the map $\widehat{\Hilb^2(S)}\to F(X)$, the image of $\hat{C_t}$ under the blow-down remains non-rational. So we have the following theorem:
\begin{theorem}
\label{theorem3}
Let $X$ be a smooth cubic fourfold in $\PR^5$ and let $F(X)$ be the Fano variety of lines on the cubic fourfold $X$. Let the cubic $X$ is such that there exists a birational map from the Hilbert scheme of two points on a fixed K3 surface to the Fano variety of $X$ and let the indeterminacy locus of the rational map $\Hilb^2 S\to F(X)$ is smooth. Then there exists a non-rational curve $D$ in $F(X)$ such that the push-forward induced by the inclusion of $D$ into $F(X)$ has torsion kernel.
\end{theorem}

\subsection{Relation of the above main theorem with the work of Voisin}
\label{Subsec1}
In the recent work due to Voisin, \cite{Vo}, it has been proved that a cubic fourfold admits of a Chow theoretic decomposition of the diagonal if and only if it admits a cohomological decomposition of the diagonal. By Chow theoretic decomposition we mean that the diagonal on the two fold product on the cubic $X$ is rationally equivalent to $X\times x+Z$, where $Z$ is supported on $D\times X$, where $D$ is a proper closed subscheme of $X$. Cohomological decomposition means that such a decomposition holds at the level of Cohomology. It has been proved in \cite{Vo} that the Chow theoretic decomposition of the diagonal of $X$ is equivalent to have universal triviality of $\CH_0(X)$, meaning that $\CH_0(X_L)$ is isomorphic to $\mathbb Z$ for all field extensions $L$ of $k$.

\medskip

Let $D$ be the smooth projective surface inside $F(X)$, such that the push-forward induced by the closed embedding $j$ of $D$ into $F(X)$ is torsion. This torsionness actually follows from the fact that the cycle $(j\times id)\Gamma'''-d\Delta$ is rationally equivalent to $\hat{E}_2+\hat{E_1}+\hat{Y_1}$ , where this later cycle is supported on $Z\times D$ or $D\times Z'$, where $ Z,Z'$ are proper Zariski closed subsets in $D$. Let us call $(j\times id)^*(\Gamma''')=\Gamma$, consider its image under the homomorphism $\CH^2(D\times D)\to \CH_0(D_{k(D)})$, call it $\Gamma_{k(D)}$. Consider also the image of the diagonal under this homomorphism and denote it by $\delta$. Suppose that $\CH_0(D_{k(D)})$ is isomorphic to $\ZZ$. Then
$$\Gamma_{k(D)}-d\delta$$
is rationally equivalent to $nx_{k(D)}$, where $x$ is a closed $k$-point on $X$. The above would mean that the restriction of the cycle
$$\Gamma-d\Delta$$
is rationally equivalent to the restriction of $X\times x$ on $U\times X$, for some Zariski open $U$ inside $X$. Therefore
$$\Gamma-d\Delta$$
is rationally equivalent to $X\times x+Z$, where $Z$ is supported on $D\times X$, where $D$ is the complement of $U$ in $X$. Therefore the universal triviality of $\CH_0(D)$ is giving the decomposition of $\Gamma-d\Delta$, which further gives the torsionness of the kernel of $j_*$.

Therefore we have the following theorem:

\begin{theorem}
Universal triviality of $\CH_0(D)$ and the assumption that the Fano variety $F(X)$ is birational to the Hilbert scheme of two points on a K3 surface, implies the torsionness of the kernel of $j_*$, where $j$ is the closed embedding of $D$ into $F(X)$.
\end{theorem}

Therefore if we can prove that $\ker(j_*)$ is non-torsion then that would imply the $\CH_0(D)$ is not universally trivial or the Fano variety is not birational to the Hilbert scheme of two points on a K3 surface.

It is worthwhile to mention that, in the paper \cite{MS}, Shen proved that the universal line correspondence $P$ on $F(X)\times X$ gives rise to surjective homomorphism from $\CH_0(F(X)_L)$ to $\CH_1(X_L)$ for all field extensions $L/k$. Here we prove something similar. Consider the base change of $D,F(X)$, by $k(D)$, then as before we can prove that the embedding of $D_{k(D)}$ into $F(X)_{k(D)}$ has torsion kernel at the Chow group of zero cycles. This follows from the observation that
$$(j\times id)^*(\Gamma''')_{k(D)}=(j_{k(D)}\times \id_{k(D)})^*(\Gamma'''_{k(D)})=d\Delta_{k(D)}-Y_{k(D)}$$
where $Y$ is supported on $Z\times D$ or $D\times Z'$. So we can say that the push-forward has torsion kernel universally.

\end{document}